\documentclass{amsart}
\usepackage{amssymb}
\usepackage{enumerate}
\usepackage{xcolor}
\usepackage{commath}
\usepackage{amsfonts}
\usepackage{amsmath,amsthm}
\usepackage{hyperref}
\usepackage{algorithm}
\usepackage{framed}
\usepackage{boites}
\usepackage{mathrsfs}
\usepackage[noend]{algpseudocode}
\usepackage{graphicx}

\newcommand{\pa}{\partial}

\providecommand{\norm}[1]{\l#1\|}

\newcommand{\e}{\varepsilon}

\newcommand{\F}{\mathcal{F}}

\newtheorem{thm}{Theorem}
\newtheorem{lem}[thm]{Lemma}
\newtheorem{prop}[thm]{Proposition}

\newtheorem{corollary}[thm]{Corollary}
\newtheorem{problem}{Inverse problem}
\theoremstyle{definition}

\theoremstyle{remark}
\newtheorem{rmk}[thm]{Remark}

\title{On an inverse Robin spectral problem}

\author{Matteo Santacesaria}
\address{MaLGa Center, Department of Mathematics, University of Genoa, Via Dodecaneso 35, 16146 Genova, Italy.}
\email{matteo.santacesaria@unige.it}

\author{Toshiaki Yachimura}
\address{Research Center for Pure and Applied Mathematics, Graduate School of Information Sciences, Tohoku University, Sendai 980-8579, Japan.}
\email{yachimura@ims.is.tohoku.ac.jp}


\thanks{M.S. is member of the Gruppo Nazionale per l'Analisi Matematica,
la Probabilit\`a e le loro Applicazioni (GNAMPA) of the Istituto Nazionale di Alta Matematica (INdAM) and he is partially supported by a INdAM -- GNAMPA Project 2019. M.S. carried out part of the work at the Machine Learning Genoa (MaLGa) center, Universit\`a di Genova (IT)  T. Y. was partially supported by the Grant-in-Aid for JSPS Fellows No.19J12344.}


\begin{document}

\begin{abstract}
We consider the problem of the recovery of a Robin coefficient on a part $\gamma \subset \partial \Omega$ of the boundary of a bounded domain $\Omega$ from the principal eigenvalue and the boundary values of the normal derivative of the principal eigenfunction of the Laplace operator with Dirichlet boundary condition on $\partial \Omega \setminus \gamma$. We prove uniqueness, as well as local Lipschitz stability of the inverse problem. Moreover, we present an iterative reconstruction algorithm with numerical computations in two dimensions showing the accuracy of the method.
\end{abstract}

\maketitle

\section{Introduction}
Let $\Omega \subset \mathbb{R}^n$ $(n \geq 2)$ be a bounded connected domain with boundary $\pa \Omega$ of class $C^2$, and $\gamma$, $\Gamma_D$ be disjoint nonempty closed subsets of the boundary $\pa \Omega$ such that $\pa \Omega = \Gamma_D \cup \gamma$. Let $h \in C^0(\gamma)$ and $h > 0$. In this paper, we consider the following Robin eigenvalue problem:
\begin{equation}\label{P}
\begin{cases}
-\Delta u = \lambda u &\text{in} \,\, \Omega, \\
u = 0 \hspace{-0.1cm} &\text{on} \,\, \Gamma_D, \\
\displaystyle h u +  \pa_\nu u = 0 &\text{on} \,\, \gamma, 
\end{cases}
\end{equation}
where $\nu$ is the outward unit normal vector of $\pa \Omega$. 
In what follows, we only consider the principal eigenvalue and eigenfunction (see \cite{daners2000} for the well-posedness of problem \eqref{P}). Moreover, we assume that the principal eigenfunction is positive and it is normalized by 
\begin{equation*}
\int_{\Omega} \abs{u(h)}^2 \, dx = 1. 
\end{equation*}

Our aim of this paper is to study an inverse problem for the Robin eigenvalue problem \eqref{P}. We consider the following inverse problem: 
\begin{problem}\label{inversepb}
Recover the unknown Robin coefficient $h$ defined in the inaccessible part $\gamma$ of the boundary $\pa \Omega$ from the principal eigenvalue $\lambda(h)$ and the Neumann data $\pa_{\nu} u (h) |_{\Gamma_D}$ on the accessible part $\Gamma_D$. 
\end{problem}
The inverse problem \ref{inversepb} is closely related to the coating problem and reinforcement problem \cite{Brezis1980, friedman1980, Buttazzo1987, rosencrans2006, aslanyurek2011generalized,yachimura}. Let $D \subset \mathbb{R}^n$ $(n \geq 2)$ be a bounded domain with smooth and connected boundary $\Gamma$. For sufficiently small $\e > 0$, put 
\begin{equation*}
\Sigma_{\e} = \left\{ x \in \mathbb{R}^n \,\, | \,\, x = \xi + t p(\xi) \nu_{\Gamma}(\xi) \,\,\, \text{for} \,\,\, \xi \in \Gamma, 0 < t < \e \right\}, \quad D_{\e} = D \cup \Sigma_{\e} \cup \Gamma,
\end{equation*}
where $p$ is a given positive smooth function on $\Gamma$ and $\nu_{\Gamma}$ is the outward unit normal vector to $\Gamma$. 
Let us consider the following two-phase eigenvalue problem on $D_{\e}$:
\begin{equation}\label{Preinintro}
\begin{cases}
-\mathrm{div} \left(\sigma_{\e} \nabla \Phi \right) = \Lambda \Phi \hspace{-0.1cm} &\text{in} \,\, D_{\e}, \\
\displaystyle \Phi = 0 \, &\text{on} \, \partial D_{\e},
\end{cases}
\end{equation}
where $\sigma_{\e} = \sigma_{\e}(x) \left( x \in D_{\e} \right)$ is a piecewise constant function given by 
\begin{equation*}
\sigma_{\e}(x) = \begin{cases}
1, \quad &x \in D, \\
\e, \quad &x \in \overline{\Sigma}_{\e}. 
\end{cases}
\end{equation*}
Then, Friedman \cite{friedman1980} proved the following theorem. 
\begin{thm}[Friedman \cite{friedman1980}]\label{thmFri1}
Let $\Lambda_{1}(\e)$ be the principal eigenvalue of the eigenvalue problem \eqref{Preinintro}. Then we have  
\begin{align*}
\Lambda_{1}(\e) &= \mu_{1} + o(1) \,\,\, \text{as} \,\,\, \e \to 0, \\
\Phi_{\e} &\to u_{1} \,\,\, \text{weakly} \,\,\, \text{in} \,\,\, H^{2}(D),  
\end{align*}
where $\mu_{1}$ is the principal eigenvalue and $u_{1}$ is the principal eigenfunction of the following Robin eigenvalue problem:
\begin{equation*}
\begin{cases}
- \Delta u = \mu u \,\, &\text{in} \,\, D, \\
u + p\dfrac{\pa u}{\pa \nu_{\Gamma}} = 0 \, &\text{on} \,\, \Gamma.
\end{cases}
\end{equation*}
\end{thm}
From the point of view of the thin coating problem, the Robin coefficient $h$ is equal to $1 / p$. That is, the Robin coefficient $h$ represents the thickness of the coating. The inverse problem \ref{inversepb} can be interpreted as the question of whether the thickness can be determined when the principal eigenvalue and the Neumann data on the accessible part are given. 

We remark that the setting of the inverse problem is also similar to the detection problem of internal corrosion. Let us consider the following problem: 
\begin{equation*}
\begin{cases}
-\Delta u = 0 &\text{in} \,\, \Omega, \\
u = 0 \hspace{-0.1cm} &\text{on} \,\, \Gamma_D, \\
\partial_\nu u = \phi \hspace{-0.1cm} &\text{on} \,\, \Gamma_N, \\
\displaystyle h u +  \pa_\nu u = 0 &\text{on} \,\, \gamma, 
\end{cases}
\end{equation*}
where $\gamma, \Gamma_D,\Gamma_N$ are disjoint open subsets of $\partial \Omega$ such that $\bar \gamma \cup \bar \Gamma_D \cup \Gamma_N = \partial \Omega$ and $\phi \not \equiv 0$.
The detection problem of internal corrosion is to recover the unknown Robin coefficient $h$ defined in the inaccessible part $\gamma$ of the boundary $\pa \Omega$ from the  Dirichlet data $ u (h) |_{\Gamma_N}$ on the accessible part $\Gamma_N$. Physically speaking, the Neumann data $\phi$ is the current flux, the Robin coefficient $h$ is the corrosion and $u$ the electrostatic potential. There are many results for uniqueness, stability, and reconstruction algorithm for this inverse problem. For the details about the inverse problem, see \cite{Inglese_1997, Chaabane1999, alessandrini2003, Chaabane2003, Chaabane_2003, Choulli2004, CHAABANE2004, Jin2007, sincich2007lipschitz, Ammari2008, Hu_2015, Xu2015} and the references therein. 

To our knowledge, there are few results concerning the Robin inverse eigenvalue problem \ref{inversepb}. The papers \cite{Ammari2009-2, refId0} deal with a Robin inverse eigenvalue problem when the support of the Robin coefficient is sufficiently small and gives a non-iterative algorithm of MUSIC (multiple signal classification) types for detecting the Robin coefficient from the measurements of an eigenvalue and a Neumann data of the accessible part of the boundary. 

The purpose of this paper is twofold. First, we study uniqueness and local stability for the inverse problem \ref{inversepb}. Then, we propose a reconstruction algorithm based on a Neumann tracking type functional and show its effectiveness with numerical simulations.

The paper is organized as follows. In Section \ref{sec of uniqueness and diffe}, we prove the uniqueness of the inverse problem. Also, we give some estimates for the principal eigenfunction and prove the Fr\'echet differentiability of the eigenfunction with respect to the Robin coefficient. Moreover, we obtain a local Lipschitz stability result by the result of the Fr\'echet differentiability. 
In Section \ref{sec of diffe}, we consider a Neumann tracking functional and give a Fr\'echet derivative of the functional by using the results in Section \ref{sec of uniqueness and diffe}. In Section \ref{sec of recon}, we introduce our algorithm and show numerical reconstruction examples.

\section{Uniqueness and Differentiability}\label{sec of uniqueness and diffe}
In this section, we prove uniqueness of the inverse problem and the Fr\'echet differentiability of the solution $u(h)$ in \eqref{P} with respect to the Robin coefficient $h$. Furthermore, we show the local stability of the inverse problem by the Fr\'echet differentiability assuming $h \in C^1(\gamma)$ (see Section \ref{sec:diff}).  

\subsection{Uniqueness of the inverse eigenvalue problem}\label{subsec uni}
\begin{thm}\label{uniqueness}
Let $\Omega \subset \mathbb{R}^n$ $(n \geq 2)$ be a bounded domain with  $C^2$ boundary and $\gamma$, $\Gamma_D$ be disjoint nonempty closed subsets of the boundary $\pa \Omega$ such that $\pa \Omega = \Gamma_D \cup \gamma$. Let $(\lambda(h_j),u_{j})$ be a solution of the Robin eigenvalue problems \eqref{P}, corresponding to the Robin coefficients $h_j$, with $h_j \in C^0(\gamma)$ and $h_j > 0$ for $j=1,2$. 

If $(\lambda(h_1), \pa_{\nu} u_1 |_{\Gamma_D}) =( \lambda(h_2), \pa_{\nu} u_2 |_{\Gamma_D})$, then we have $h_{1} = h_{2}$. 
\end{thm}
\begin{proof} From \cite{daners2000} we have that problem \eqref{P} is well-posed.
Put $w:= u_{1} - u_{2}$. From the assumption $(\lambda(h_1), \pa_{\nu} u_1 |_{\Gamma_D} ) = ( \lambda(h_2), \pa_{\nu} u_2 |_{\Gamma_D} )$, we have 
\begin{equation}\label{weq1}
\begin{cases}
-\Delta w = \lambda w \hspace{-0.1cm} &\text{in} \,\, \Omega, \\
w = 0 \hspace{-0.1cm} &\text{on} \,\, \Gamma_D, \\
\pa_\nu w = 0 \hspace{-0.1cm} &\text{on} \,\, \Gamma_D, \\
h_{1}w + \pa_\nu w + (h_{1} - h_{2})u_2 = 0 \, &\text{on} \,\, \gamma. 
\end{cases}
\end{equation}
By using Holmgren's unique continuation theorem (see \cite{zbMATH06684667}), we obtain $w \equiv 0$ in $\Omega$. Hence $(h_1 - h_2) u_2 = 0$ on $\gamma$. 
Let us assume that there exists a point $x_0 \in \gamma$ such that $h_1(x_0) \neq h_2(x_0)$. Then by continuity of $h_1$ and $h_2$, there exists an open subset $U \subset \gamma$ such that $h_1 \neq h_2$ in $U$. Thus we have $u_2 = 0$ on $U$.  
Due to the boundary condition for $u_2$, we also obtain $\pa_\nu u_2 = 0$ on $U$. Hence, 
\begin{equation}
\begin{cases}
-\Delta u_2 = \lambda u_2 \hspace{-0.1cm} &\text{in} \,\, \Omega, \\
u_2 = 0 \hspace{-0.1cm} &\text{on} \,\, U, \\
\pa_\nu u_2 = 0 \hspace{-0.1cm} &\text{on} \,\, U. 
\end{cases}
\end{equation}
Applying Holmgren's theorem again, we obtain $u_2 \equiv 0$ in $\Omega$. However, this is in contradiction with $u_2 \not\equiv 0$. Hence we have $h_{1} = h_{2}$. 
\end{proof}

\begin{rmk}
We remark that Theorem \ref{uniqueness} holds when the given spectral data is not only the principal eigenvalue and the Neumann data of the principal eigenfunction on $\Gamma_D$, but also the $k$-th $(k \geq 2)$ eigenvalue and the Neumann data of the corresponding eigenfunction on $\Gamma_D$. 
\end{rmk}

\subsection{Differentiability}\label{sec:diff}
Next, we consider the Fr\'echet differentiability of the solution $u(h)$ in \eqref{P} with respect to the Robin coefficient $h$. Let $\mathscr{A}$ be the admissible set of Robin coefficients, defined by 
\begin{equation*}
    \mathscr{A} = \left\{h \in C^1(\gamma) \,:\, h(x) > 0 \right\}. 
\end{equation*}
Moreover, let us introduce the following space: 
\begin{equation*}
    V = \{ u \in H^1(\Omega) \, : \, u = 0 \,\, \text{on} \,\, \Gamma_D \}, 
\end{equation*}
and its norm
\begin{equation}\label{norm of V}
    \norm{u}^2_{V} = \int_{\Omega} \abs{\nabla u}^2 \, dx + \int_{\gamma} h \abs{u}^2 \, ds.  
\end{equation}
We remark that the norm \eqref{norm of V} is equivalent to the usual norm on $H^1(\Omega)$, thanks to Poincar\'e's inequality with trace \cite{kuznetsov2015sharp}. Thus we will use the norm \eqref{norm of V} in what follows.  
Then we can obtain the following result of the Fr\'echet differentiability of the solution $u(h)$ in \eqref{P}: 
\begin{thm}\label{differe}
Let $\Omega, \gamma, \Gamma_D$ be as in Theorem \ref{uniqueness}. Then the solution $u(h) \in V$ of \eqref{P}, with $h \in \mathscr{A}$, is Fr\'echet differentiable in the following sense:   
\begin{equation*}
    \dfrac{\norm{u(h + \xi) - u(h) - u'(h)[\xi]}_{V}}{\norm{\xi}_{C^1(\gamma)}} \to 0 \,\, \text{as} \,\, \norm{\xi}_{C^1(\gamma)} \to 0,  
\end{equation*}
where $u'(h)[\xi] \in V$ is the solution of the following sensitivity problem: 
\begin{equation}\label{sensitivity equation}
\begin{cases}
    - \Delta u' - \lambda(h) u' = \lambda' u(h) \hspace{-0.1cm} &\text{in} \,\, \Omega, \\ 
    u' = 0 \hspace{-0.1cm} &\text{on} \,\, \Gamma_D, \\ 
    hu' + \pa_\nu u' = -\xi u(h) &\text{on} \,\, \gamma, \\ 
    \displaystyle \int_{\Omega} u(h) u' \, dx = 0. 
\end{cases}
\end{equation}
Here, $\lambda'$ is given by $\displaystyle \lambda' = \int_\gamma \xi u(h)^2 \, ds$. 
\end{thm}

Before proving Theorem \ref{differe}, we need to show some estimates. In this section, $C$ will denote a positive constant $C > 0$ depending on $u$, $\gamma$ and $\Omega$.
\begin{lem}\label{esti1}
Under the assumptions of Theorem \ref{differe}, the following estimate holds: 
\begin{equation*}
    \norm{u(h + \xi) - u(h)}_{L^2(\Omega)} \to 0 \,\, \text{as} \,\, \norm{\xi}_{C^1(\gamma)} \to 0. 
\end{equation*}
\end{lem}
\begin{proof}
By the min-max principle (see \cite[\S 2]{LOTOREICHIK2017491} and \cite{BUCUR2019643} for instance), we obtain 
\begin{equation*}
    \lambda(h+\xi) = \displaystyle \inf_{\Phi \in V, \Phi \neq 0} \dfrac{\displaystyle \int_\Omega \abs{\nabla \Phi}^2 \, dx + \int_\gamma (h + \xi) \Phi^2 \, ds}{\displaystyle \int_\Omega \Phi^2 \, dx}.   
\end{equation*}
Let us take $\Phi = u(h)$, then, by using the following identity 
\begin{equation*}
    \int_\Omega \abs{\nabla u(h)}^2 \, dx + \int_\gamma h u(h)^2 \, ds = \lambda(h), 
\end{equation*}
we have 
\begin{align}\label{upper bound}
    \lambda(h+\xi) \leq \int_\Omega \abs{\nabla u(h)}^2 \, dx + \int_\gamma h u(h)^2 \, ds + \int_\gamma \xi u(h)^2 \, ds \leq \lambda(h) + C \norm{\xi}_{C^1(\gamma)}.
\end{align}

Let us consider the limit of $u(h+\xi)$ for $\norm{\xi}_{C^1(\gamma)} \to 0$. Taking  $\norm{\xi}_{C^1(\gamma)}$ sufficiently small, and using the upper bound \eqref{upper bound} and the identity 
\begin{equation*}
    \int_\Omega \abs{\nabla u(h+\xi)}^2 \, dx + \int_\gamma (h+\xi) u(h+\xi)^2 \, ds = \lambda(h+\xi), 
\end{equation*}
we can get the uniform boundedness of $u(h+\xi)$ in $V$ (and so in $H^1(\Omega)$). Applying Rellich's Theorem, taking subsequences, there exist $\Hat{\lambda}$ and $\Hat{u} \in V$ such that 
\begin{equation*}
    \begin{cases}
        \lambda(h + \xi) \to \Hat{\lambda}, \\
        u(h + \xi) \to \Hat{u} \quad \text{strongly in} \, L^2(\Omega), \\
        u(h + \xi) \rightharpoonup \Hat{u} \quad \text{weakly in} \, V. 
    \end{cases}
\end{equation*}
Then from the weak form of $u(h+\xi)$, taking the limits, we have 
\begin{equation}\label{eqhat}
    \int_{\Omega} \nabla \Hat{u} \cdot \nabla \phi \, dx + \int_\gamma h \Hat{u} \phi \, ds = \Hat{\lambda} \int_{\Omega} \Hat{u} \phi \, dx
\end{equation}
for any $\phi \in V$. If we take $\phi = \Hat{u}$ and use min-max principle, we get $\Hat{\lambda} \leq \lambda(h)$. Moreover, we consider min-max principle for the principal eigenvalue $\lambda(h)$
\begin{equation*}
    \lambda(h) = \displaystyle \inf_{\Phi \in V, \Phi \neq 0} \dfrac{\displaystyle \int_\Omega \abs{\nabla \Phi}^2 \, dx + \int_\gamma h \Phi^2 \, ds}{\displaystyle \int_\Omega \Phi^2 \, dx}  
\end{equation*}
and take $\Phi = \Hat{u}$, by \eqref{eqhat}, we obtain $\Hat{\lambda} = \lambda(h)$. Since $\lambda(h)$ is the principal eigenvalue, we have $\Hat{u} = u(h)$. Therefore, 
\begin{align*}
    \norm{u(h+\xi) - u(h)}_{L^2(\Omega)} \to 0,
\end{align*}
as $\|\xi\|_{C^1(\gamma)} \to 0$.
\end{proof}

Lemma \ref{esti1} gives us the following estimates: 
\begin{prop}\label{prop:lambda}
Under the assumptions of Theorem \ref{differe}, we have the following eigenvalue estimate:
\begin{equation}\label{estieigen}
    \lambda(h+\xi) = \lambda(h) + \dfrac{\displaystyle \int_\gamma \xi u(h+\xi)u(h) \, ds}{\displaystyle \int_\Omega u(h + \xi) u(h) \, dx}. 
\end{equation}
Moreover, the following estimate holds:
\begin{equation}\label{H^1 estimate}
    \norm{u(h+\xi) - u(h)}_{V} = o(\norm{\xi}^{1/2}_{C^1(\gamma)}) \,\, \text{as} \,\, \norm{\xi}_{C^1(\gamma)} \to 0. 
\end{equation}
\end{prop}

\begin{proof}
Let us put $v = u(h+\xi) - u(h)$. Then $v$ satisfies 
\begin{equation}\label{v}
\begin{cases}
-\Delta v - \lambda(h)v = \left( \lambda(h+\xi) - \lambda(h) \right) u(h+\xi) &\text{in} \,\, \Omega, \\
v = 0 \hspace{-0.1cm} &\text{on} \,\, \Gamma_D, \\
 h v +  \pa_\nu v = -\xi u(h+\xi) &\text{on} \,\, \gamma. 
\end{cases}
\end{equation}
Let us consider the following weak form of $v$: 
\begin{multline}\label{weakform v}
    \int_{\Omega} \nabla v \cdot \nabla \phi \, dx + \int_\gamma h v \phi \, ds - \lambda(h) \int_{\Omega} v \phi \, dx \\
    = \left( \lambda(h+\xi) - \lambda(h) \right) \int_{\Omega} u(h+\xi) \phi \, dx - \int_\gamma \xi u(h+\xi) \phi \, ds \quad \forall \phi \in V,   
\end{multline}
Since $\lambda(h)$ is the principal eigenvalue, due to Fredhom alternative for \eqref{weakform v} (see \cite[Theorem 7.44.2, p.648]{medkova2018}), we have 
\begin{equation}\label{res Fred}
    \left( \lambda(h+\xi) - \lambda(h) \right) \int_{\Omega} u(h+\xi) u(h) \, dx = \int_\gamma \xi u(h+\xi) u(h) \, ds.
\end{equation}
Therefore,  
\begin{equation*}
    \lambda(h+\xi) = \lambda(h) + \dfrac{\displaystyle \int_\gamma \xi u(h+\xi)u(h) \, ds}{\displaystyle \int_\Omega u(h + \xi) u(h) \, dx}. 
\end{equation*}
Note that from Lemma \ref{esti1} and the normalization of the principal eigenfunction $u(h)$, 
\begin{align*}
    \left| \int_\Omega u(h + \xi) u(h) \, dx - 1 \right| &= \left| \int_\Omega u(h + \xi) u(h) \, dx - \int_\Omega u(h)^2 \, dx \right| \\
    &\leq \norm{u(h+\xi) - u(h)}_{L^2(\Omega)} \to 0,  
\end{align*}
as $\norm{\xi}_{C^1(\gamma)} \to 0$. Thus, we have     
\begin{equation}\label{innerpro1}
    \int_\Omega u(h + \xi) u(h) \, dx = 1 + o(1) \,\, \text{as} \,\, \norm{\xi}_{C^1(\gamma)} \to 0. 
\end{equation}
By using \eqref{innerpro1}, we obtain 
\begin{align}\label{order eigen1}
    \lambda(h+\xi) &= \lambda(h) + \dfrac{\displaystyle \int_\gamma \xi u(h+\xi)u(h) \, ds}{\displaystyle \int_\Omega u(h + \xi) u(h) \, dx} \notag \\
    &= \lambda(h) + \dfrac{\displaystyle \int_\gamma \xi u(h+\xi)u(h) \, ds}{1 + o(1)} \notag \\
    &= \lambda(h) + O(\norm{\xi}_{C^1(\gamma)}). 
\end{align}

Next we will consider the estimate for $v$. If we take $\phi = v$ as a test function in \eqref{weakform v}, we have  
\begin{multline}\label{temp1}
    \int_{\Omega} \abs{\nabla v}^2 \, dx + \int_\gamma h v^2 \, ds - \lambda(h) \int_\Omega v^2 \, dx \\
    = \left( \lambda(h+\xi) - \lambda(h) \right) \int_\Omega u(h+\xi) v \, dx - \int_\gamma \xi u(h + \xi) v \, ds. 
\end{multline}
From Lemma \ref{esti1} and \eqref{order eigen1},  
\begin{equation}\label{temp2}
    \norm{v}^2_{V} \leq \lambda(h) \norm{v}^2_{L^2(\Omega)} + C \norm{\xi}_{C^1(\gamma)} \norm{v}^2_{L^2(\Omega)} + C \norm{\xi}_{C^1(\gamma)} \to 0 \,\, \text{as} \,\, \norm{\xi}_{C^1(\gamma)} \to 0. 
\end{equation}

In order to get further estimate of $v$, let us consider the Fourier expansions of $v$ with respect to the mixed eigenvalue problem as follows: 
\begin{equation}\label{eigenvaluepro}
\begin{cases}
-\Delta u_k = \lambda_k u_k &\text{in} \,\, \Omega, \\
u_k = 0 \hspace{-0.1cm} &\text{on} \,\, \Gamma_D, \\
 h u_k +  \pa_\nu u_k = 0 &\text{on} \,\, \gamma,  
\end{cases}
\end{equation}
where $\{ \lambda_k \}_{k \geq 1}$ are eigenvalues and $\{ u_k \}_{k \geq 1}$ are corresponding normalized eigenfunctions. Note that $\lambda_1 = \lambda(h)$ and $u_1 = u(h)$. Then $v$ admits the following Fourier expansions in $V$ (for the proof, see \cite[Theorem 7.44.1, p.646]{medkova2018} and \cite[Section $6.5$, Theorem $2$]{Evans} for instance): 
\begin{equation}\label{Fourier of v}
    v = \sum_{k \geq 1} \alpha_k u_k, \quad \alpha_k = \int_{\Omega} v u_k \, dx. 
\end{equation}
By the Fourier expansions \eqref{Fourier of v} and the orthogonality of eigenfunctions in \eqref{eigenvaluepro}, the left-hand side of \eqref{temp1} is
\begin{align*}
    &\int_{\Omega} \abs{\nabla v}^2 \, dx + \int_\gamma h v^2 \, ds - \lambda(h) \int_\Omega v^2 \, dx \\
    &= \sum_{k_1,k_2 \geq 1} \alpha_{k_1} \alpha_{k_2} \left(\int_\Omega \nabla u_{k_1} \cdot \nabla u_{k_2} \, dx + \int_\gamma h u_{k_1} u_{k_2} \, ds \right) - \lambda_1 \sum_{k \geq 1} \alpha_k^2 \\
    &= \sum_{k \geq 1} \lambda_k \alpha_k^2 - \lambda_1 \sum_{k \geq 1} \alpha_k^2 \\ 
    &= \sum_{k \geq 1} (\lambda_k - \lambda_1) \alpha_k^2 = \sum_{k \geq 2} \lambda_k \alpha_k^2. 
\end{align*}
Also, by using \eqref{estieigen}, the right-hand side of \eqref{temp1} is 
\begin{align*}
    &\left( \lambda(h+\xi) - \lambda(h) \right) \int_\Omega u(h+\xi) v \, dx - \int_\gamma \xi u(h + \xi) v \, ds \\
    &= \dfrac{\displaystyle \int_\gamma \xi u(h+\xi)u(h) \, ds}{\displaystyle \int_\Omega u(h + \xi) u(h) \, dx} \left( 1 - \int_\Omega u(h+\xi)u(h) \, dx \right) \\
    &\qquad- \int_\gamma \xi u(h+\xi) \left( u(h+\xi) - u(h) \right) \, ds \\
    &= \dfrac{\displaystyle \int_\gamma \xi u(h+\xi)u(h) \, ds}{\displaystyle \int_\Omega u(h + \xi) u(h) \, dx} - \int_\gamma \xi u(h+\xi)^2 \, ds. 
\end{align*}
By using \eqref{innerpro1}, \eqref{temp2}, and the trace theorem, we get 
\begin{align*}
    &\dfrac{\displaystyle \int_\gamma \xi u(h+\xi)u(h) \, ds}{\displaystyle \int_\Omega u(h + \xi) u(h) \, dx} - \int_\gamma \xi u(h+\xi)^2 \, ds \\
    &= \dfrac{\displaystyle \int_\gamma \xi u(h+\xi)u(h) \, ds}{1+o(1)} - \int_\gamma \xi u(h+\xi)^2 \, ds \\
    &= \int_\gamma \xi u(h+\xi)u(h) \, ds - \int_\gamma \xi u(h+\xi)^2 \, ds + o(\norm{\xi}_{C^1(\gamma)}) \\ 
    &= \int_\gamma \xi u(h+\xi)(u(h) - u(h+\xi)) \, ds + o(\norm{\xi}_{C^1(\gamma)}) = o(\norm{\xi}_{C^1(\gamma)}).
\end{align*}
Combining these estimates, we have 
\begin{equation}\label{temp3}
    \sum_{k \geq 2} \lambda_k \alpha_k^2 = o(\norm{\xi}_{C^1(\gamma)}). 
\end{equation}
Let us put $\delta = \sum_{k \geq 2} \alpha^2_k$. Then it can be estimated easily as 
\begin{equation*}
    \lambda_2 \delta \leq \sum_{k \geq 2} \lambda_k \alpha_k^2 = o(\norm{\xi}_{C^1(\gamma)}). 
\end{equation*}
Thus we obtain $\delta = o(\norm{\xi}_{C^1(\gamma)})$. Moreover, 
\begin{align*}
    \sum_{k \geq 1} \alpha^2_k = \int_\Omega v^2 \, dx &= \int_\Omega \left( u(h+\xi) - u(h) \right)^2 \, dx \\
    &= 2 - 2 \int_\Omega u(h+\xi) u(h) \, dx \\
    &= 2 \int_\Omega u(h)^2 \, dx - 2 \int_\Omega u(h+\xi) u(h) \, dx \\ 
    &= -2 \int_\Omega v u(h) \, dx = -2\alpha_1. 
\end{align*}
Thus we have $\alpha^2_1 + 2\alpha_1 + \delta = 0$. Note that $\alpha_1$ is small by Lemma \ref{esti1}. Since $\delta$ is sufficiently small, we have 
\begin{align}\label{esti fourier for alpha_1}
    \alpha_1 &= -1 + \sqrt{1-\delta} \notag \\
    &= -\frac{1}{2} \delta + o(\norm{\xi}^2_{C^1(\gamma)}). 
\end{align}
It follows that, by \eqref{temp3}, 
\begin{equation*}
    \norm{v}^2_{V} = \sum_{k \geq 1} \lambda_k \alpha^2_k = \lambda_1 \alpha^2_1 + \sum_{k \geq 2} \lambda_k \alpha_k^2 = o(\norm{\xi}_{C^1(\gamma)}). 
\end{equation*}
Therefore we obtain $\norm{v}_{V} = o(\norm{\xi}^{1/2}_{C^1(\gamma)})$, which finishes the proof. 
\end{proof}

Next, let us show the smallness of $u'$. 
\begin{prop}\label{H^1 estimate for u'}
Let $u' \in V$ be the solution of \eqref{sensitivity equation}. Then the following estimate holds:
\begin{equation}\label{H^1 estimate for $u'$}
     \norm{u'}_{V} = O(\norm{\xi}_{C^1(\gamma)}) \,\, \text{as} \,\, \norm{\xi}_{C^1(\gamma)} \to 0. 
\end{equation}
\end{prop}

\begin{proof}
Existence and uniqueness for \eqref{sensitivity equation} follows from \cite{daners2000}.
Consider the weak form of the sensitivity equation \eqref{sensitivity equation} 
\begin{equation*}
    \int_{\Omega} \nabla u' \cdot \nabla \phi \, dx + \int_\gamma h u' \phi \, ds = \lambda(h) \int_{\Omega} u' \phi \, dx + \lambda' \int_{\Omega} u(h) \phi \, dx - \int_\gamma \xi u(h) \phi \, ds 
\end{equation*}
for all $\phi \in V$, where $\displaystyle \lambda' = \int_\gamma \xi u(h)^2 \, ds$. Taking $\phi = u'$ and using the assumption $\displaystyle \int_{\Omega} u(h) u' \, dx = 0$, we have 
\begin{equation}\label{energy for u'}
   \int_\Omega \abs{\nabla u'}^2 \, dx + \int_\gamma h (u')^2 \, ds = \lambda(h) \int_\Omega (u')^2 \, dx - \int_\gamma \xi u u' \, ds. 
\end{equation}
Let us consider the following Fourier expansions for $u'$: 
\begin{equation}\label{Fourier of u'}
    u' = \sum_{k \geq 1} \beta_k u_k, \quad \beta_k = \int_{\Omega} u' u_k \, dx,
\end{equation}
where the functions $u_k$ were defined in \eqref{eigenvaluepro}.
By the assumption $\displaystyle \int_{\Omega} u(h) u' \, dx = 0$, we obtain $\beta_1(\xi) \equiv 0$. By substituting \eqref{Fourier of u'} into \eqref{energy for u'}, then in the same manner of the estimate for $v$, we obtain  
\begin{equation*}
    \sum_{k \geq 2} \lambda_k(h) \beta_k^2 = \lambda_1(h) \sum_{k \geq 2} \beta_k^2 - \int_\gamma \xi u u' \, ds. 
\end{equation*} 
Note that by using the trace theorem we have 
\begin{align*}
    \abs{\int_\gamma \xi u(h) u' \, ds} &\leq \norm{\xi}_{C^1(\gamma)} \norm{u(h)}_{L^2(\gamma)} \norm{u'}_{L^2(\gamma)} \\
    &\leq C\norm{\xi}_{C^1(\gamma)} \norm{u'}_{L^2(\gamma)} \\
    &\leq C\norm{\xi}_{C^1(\gamma)} \norm{u'}_{V}. 
\end{align*}
Thus we obtain 
\begin{equation*}
    (\lambda_2(h) -\lambda_1(h)) \sum_{k \geq 2} \beta_k^2\leq  \sum_{k \geq 2} (\lambda_k(h) -\lambda_1(h))\beta_k^2  \leq C\norm{\xi}_{C^1(\gamma)} \norm{u'}_{V}. 
\end{equation*}
It follows that 
\begin{equation}\label{comp1}
    \norm{u'}^2_{L^2(\Omega)} \leq C\norm{\xi}_{C^1(\gamma)} \norm{u'}_{V}. 
\end{equation}
Furthermore, by \eqref{energy for u'} we have  
\begin{equation}\label{comp2}
    \norm{u'}^2_{V} \leq \lambda_1(h) \norm{u'}^2_{L^2(\Omega)} + C \norm{\xi}_{C^1(\gamma)} \norm{u'}_{V}.
\end{equation}
Combining \eqref{comp1} and \eqref{comp2}, we can show that  $\norm{u'}_{V} = O(\norm{\xi}_{C^1(\gamma)})$. 
\end{proof}

Now we can prove Theorem \ref{differe}. 
\begin{proof}[Proof of Theorem \ref{differe}]
Let us put $w = u(h+\xi) - u(h) - u'(h)[\xi]$. Then $w$ satisfies 
\begin{equation}\label{eigenvaluepro for w}
\begin{cases}
-\Delta w - \lambda(h) w = (\lambda(h+\xi) - \lambda(h))u(h+\xi) - \lambda'u(h) &\text{in} \,\, \Omega, \\
w = 0 \hspace{-0.1cm} &\text{on} \,\, \Gamma_D, \\
 h w +  \pa_\nu w = - \xi (u(h+\xi) - u(h)) &\text{on} \,\, \gamma.   
\end{cases}
\end{equation}
Consider the weak form of \eqref{eigenvaluepro for w} 
\begin{multline*}
    \int_{\Omega} \nabla w \cdot \nabla \phi \, dx + \int_\gamma h w \phi \, ds - \lambda(h) \int_\Omega w \phi \, dx \\
    = (\lambda(h+\xi) - \lambda(h)) \int_{\Omega} u(h+\xi) \phi \, dx - \lambda' \int_{\Omega} u(h) \phi \, dx - \int_\gamma \xi (u(h+\xi) - u(h)) \phi \, ds  
\end{multline*}
for all $\phi \in V$. Since $\lambda(h)$ is the principal eigenvalue of the eigenvalue problem \eqref{P}, by using Fredhom alternative in the same manner of \eqref{res Fred}, we have 
\begin{equation*}
    (\lambda(h+\xi) - \lambda(h)) \int_{\Omega} u(h+\xi) u(h) \, dx - \lambda' - \int_\gamma \xi (u(h+\xi) - u(h)) u(h) \, ds = 0. 
\end{equation*}
By the normalization of $u(h)$, we obtain 
\begin{align*}
    &\lambda(h+\xi) - \lambda(h) - \lambda' \\
    &= (\lambda(h) - \lambda(h+\xi)) \int_{\Omega} (u(h+\xi) - u(h)) u(h) \, dx + \int_\gamma \xi (u(h+\xi) - u(h)) u(h) \, ds \\ 
    &= (\lambda(h) - \lambda(h+\xi)) \int_{\Omega} v u(h) \, dx + \int_\gamma \xi v u(h) \, ds. 
\end{align*}
Thus by Proposition \ref{prop:lambda} (in particular \eqref{order eigen1}) and the trace theorem we obtain 
\begin{align}\notag
    \abs{\lambda(h+\xi) - \lambda(h) - \lambda'} &\leq C \norm{\xi}_{C^1(\gamma)} \times o(\norm{\xi}_{C^1(\gamma)}^{1/2}) + C \norm{\xi}_{C^1(\gamma)} \times o(\norm{\xi}_{C^1(\gamma)}^{1/2})\\ \label{eigenvalue esti}
     &= o(\norm{\xi}_{C^1(\gamma)}^{3/2}). 
\end{align}
In the same way of the proof of Proposition \ref{prop:lambda} and Proposition \ref{H^1 estimate for u'}, we take $\phi = w$ and consider the Fourier expansions as follows: 
\begin{equation}\label{Fourier of w}
    w = \sum_{k \geq 1} c_k u_k, \quad c_k = \int_{\Omega} w u_k \, dx. 
\end{equation}
Then, in the same manner of the estimates for $v$ and $u'$, we obtain 
\begin{align*}
    \sum_{k \geq 2} \lambda_k(h) c_k^2 &= (\lambda(h+\xi) - \lambda(h) - \lambda') \int_\Omega u(h)v \, dx + (\lambda(h+\xi) - \lambda(h)) \int_\Omega v^2 \, dx \\
    &\qquad - (\lambda(h+\xi) - \lambda(h)) \int_\Omega vu' \, dx - \int_\gamma \xi v^2 \, ds + \int_\gamma \xi v u' \, ds \\
    &\leq o(\norm{\xi}_{C^1(\gamma)}^{3/2}) \times o(\norm{\xi}_{C^1(\gamma)}^{1/2}) + O(\norm{\xi}_{C^1(\gamma)}) \times o(\norm{\xi}_{C^1(\gamma)}) \\
    &\qquad + O(\norm{\xi}_{C^1(\gamma)}) \times o(\norm{\xi}_{C^1(\gamma)}^{1/2}) \times O(\norm{\xi}_{C^1(\gamma)})\\
    &\qquad + O(\norm{\xi}_{C^1(\gamma)}) \times o(\norm{\xi}_{C^1(\gamma)}) \\
    &\qquad  + O(\norm{\xi}_{C^1(\gamma)}) \times o(\norm{\xi}_{C^1(\gamma)}^{1/2}) \times O(\norm{\xi}_{C^1(\gamma)}) \\
    &= o(\norm{\xi}_{C^1(\gamma)}^{2}). 
\end{align*}
Furthermore, by \eqref{esti fourier for alpha_1}, 
\begin{equation*}
    c_1 = \int_\Omega w u(h) \, dx = \int_\Omega v u(h) \, dx = \alpha_1 = o(\norm{\xi}_{C^1(\gamma)}). 
\end{equation*}
Therefore, we have 
\begin{equation*}
    \sum_{k \geq 1} \lambda_k c_k^2 = \lambda_1(h) c_1^2 + \sum_{k \geq 2} \lambda_k(h) c_k^2 = o(\norm{\xi}_{C^1(\gamma)}^{2}). 
\end{equation*}
It follows that $\norm{w}_{V} = o(\norm{\xi}_{C^1(\gamma)})$, which is the desired conclusion.
\end{proof}
\begin{rmk}
We remark that Theorem \ref{differe} also holds when $\norm{\xi}_{C^0(\gamma)}$ is replaced by $\norm{\xi}_{C^1(\gamma)}$. We need the assumption that $h$ (also $\xi$) is of class $C^1(\gamma)$ in order to prove Corollary \ref{differe H^2}. 
\end{rmk}
By Theorem \ref{differe} and the elliptic regularity theory for mixed boundary problem (see \cite[Theorem 7.36.6, p.621]{medkova2018}), we can prove the Fr\'echet differentiability for $u(h)$ in $H^2$. 
\begin{corollary}\label{differe H^2}
The solution $u(h) \in V$ of \eqref{P}, with $h \in \mathscr{A}$, is Fr\'echet differentiable in the following sense:
\begin{equation*}
    \dfrac{\norm{u(h + \xi) - u(h) - u'(h)[\xi]}_{H^2(\Omega)}}{\norm{\xi}_{C^1(\gamma)}} \to 0 \,\, \text{as} \,\, \norm{\xi}_{C^1(\gamma)} \to 0.  
\end{equation*}
\end{corollary}
\begin{proof}
Let $f \in L^2(\Omega)$ and $g \in H^{1/2}(\gamma)$. Then by the standard elliptic regularity theory, there exists a unique solution $w \in H^2(\Omega)$ of the mixed boundary value problem 
\begin{equation*}
\begin{cases}
-\Delta w = f &\text{in} \,\, \Omega, \\
w = 0 \hspace{-0.1cm} &\text{on} \,\, \Gamma_D, \\
\displaystyle h w +  \pa_\nu w = g &\text{on} \,\, \gamma 
\end{cases}
\end{equation*}
with $H^2$-estimate
\begin{equation}\label{estimate for w in H^2}
    \norm{w}_{H^2(\Omega)} \leq C \left( \norm{f}_{L^2(\Omega)} + \norm{g}_{H^{1/2}(\gamma)} \right), 
\end{equation}
where the constant $C$ does not depend on $f$ and $g$. In the estimate \eqref{estimate for w in H^2}, if we take $f = \lambda(h)w + (\lambda(h+\xi) - \lambda(h))u(h+\xi) - \lambda' u(h)$ and $g = -\xi (u(h+\xi) - u(h))$, then by the estimates \eqref{H^1 estimate}, \eqref{eigenvalue esti}, and Theorem \ref{differe}, we can obtain $\norm{w}_{H^2(\Omega)} = o(\norm{\xi}_{C^1(\gamma)})$. 
\end{proof}

\subsection{Local Lipschitz stability}

From the results of this section, it is now possible to derive a local stability (or local injectivity) result.

\begin{thm}
Let $h, \xi \in \mathscr{A}$. Then
\begin{align}\notag
\lim_{\varepsilon\to 0} &\left(\frac{\|\partial_\nu u(h+\varepsilon\xi)-\partial_\nu u(h)\|_{L^{2}(\Gamma_D)}}{|\varepsilon|}\right.\\ \label{stability}
&\quad +\left.\frac{\abs{\lambda(h+\varepsilon \xi) - \lambda(h)}}{|\varepsilon|}\right) > 0.
\end{align}
\end{thm}

\begin{proof}
For $|\varepsilon|$ sufficiently small, $h+\varepsilon\xi \in \mathscr{A}$. Then we can apply the results of the previous section, in particular Corollary \ref{differe H^2} and \eqref{eigenvalue esti}, which show that condition \eqref{stability} is equivalent to 
\begin{equation}
\|\partial_\nu u'(h)[\xi]\|_{L^{2}(\Gamma_D)} + |\lambda'(h)[\xi]| >0.
\end{equation}
Assume by contradiction that $\|\partial_\nu u'(h)[\xi]\|_{L^{2}(\Gamma_D)} + |\lambda'(h)[\xi]| =0$, so $\partial_\nu u'(h)[\xi]= 0$ on $\Gamma_D$ and $\lambda'(h)[\xi] = 0$. By \eqref{sensitivity equation} we have that $u' = u'(h)[\xi]$ solves
\begin{equation}\notag
\begin{cases}
    - \Delta u' - \lambda(h) u' =0 \hspace{-0.1cm} &\text{in} \,\, \Omega, \\ 
    \partial_\nu u' = u' = 0 \hspace{-0.1cm} &\text{on} \,\, \Gamma_D, \\ 
    \displaystyle \int_{\Omega} u(h) u' = 0,
\end{cases}
\end{equation}
thus by Holmgren's theorem $u'(h)[\xi] \equiv 0$ in $\Omega$. Again thanks to \eqref{sensitivity equation} we find that
\[
0 = hu' + \pa_\nu u' = -\xi u \quad \text{on } \gamma.
\]
From the continuity of $\xi$ and the fact that $\xi \not \equiv 0$ on $\gamma$ there is an open subset $V$ of $\gamma$ where $u = 0$. Now the Robin boundary condition for $u$ yields $\partial_\nu u = 0$ on $V$ and by Holmgren's theorem again we find that $u \equiv 0$ in $\Omega$, which is impossible since the principal eigenfunction is not identically zero.
\end{proof}

This result does not yield a Lipschitz stability estimate immediately. For that we would need to impose further assumptions, for instance that $h$ belongs to a finite-dimensional subspace of $\mathcal{A}$. Moreover, we also need to show that the forward map $h \mapsto (\lambda(h),\partial_\nu u|_{\Gamma_D})$ is $C^1$, which follows from the min-max principle for the eigenvalue and the well-posedness of the mixed boundary value problem for the eigenfunction.

Note that, by using similar arguments as in \cite{alberti2019}, it would be possible to obtain a local Lipschitz stability (and also local uniqueness) when only a discretization of the Neumann data on the boundary is available (under the assumptions that $h$ belong to a known finite-dimensional subspace). We also expect that a global logarithmic stability estimate holds, as in the problem of the determination of corrosion \cite{alessandrini2003, Chaabane_2003}.

\section{Neumann tracking type functional and its properties}\label{sec of diffe}
In this section, let us introduce a Neumann tracking type functional in order to turn the inverse problem into an optimization problem. 

We consider the following least-squares functional $\mathcal{F}$ over the admissible set $\mathscr{A}$ defined by:
\begin{equation}\label{cost functional}
    \mathcal{F}(h) = \dfrac{1}{2} \int_{\Gamma_D} (\pa_\nu u(h) - g)^2 \, ds + \dfrac{1}{2} \abs{\lambda(h) - \lambda}^2,    
\end{equation}
where $(\lambda, g)$ are the given spectral data. Also let us consider a Tikhonov regularization functional $\mathcal{F}_\text{reg}$ for the functional $\mathcal{F}$ defined by
\begin{equation}\label{regularized cost functional}
    \mathcal{F}_\text{reg}(h) = \dfrac{1}{2} \int_{\Gamma_D} (\pa_\nu u(h) - g)^2 \, ds + \dfrac{1}{2} \abs{\lambda(h) - \lambda}^2 + \dfrac{\eta}{2} \int_{\gamma} h^2\, ds, 
\end{equation}
where $\eta > 0$ is a regularization parameter. The choice of the $L^2$ regularization has been motivated by the numerical results and the quick algorithmic  implementation. We leave the analysis and simulation of more advanced regularizers to future work.

By Theorem \ref{uniqueness}, we can easily show that the functional \eqref{cost functional} has a unique minimizer in $\mathscr{A}$ which is the solution of the inverse problem.  
\begin{prop}
There exists a unique function $h \in \mathscr{A}$ of the functional \eqref{cost functional} such that 
\begin{equation*}
    0 = \mathcal{F}(h) \leq \mathcal{F}(\psi) \quad \forall \psi \in \mathscr{A}. 
\end{equation*}
Moreover $h$ is the solution of the inverse problem. 
\end{prop}
\begin{proof}
Let $h$ be the solution of the inverse problem. Then we obtain $\lambda(h) = \lambda$ and $\pa_\nu u|_{\Gamma_D} = g$. Thus $h$ is a minimum for $\mathcal{F}$ with $\mathcal{F}(h) = 0$. On the other hand, we assume that $\mathcal{F}(h) = 0$. Then we can easily see that $h$ is the solution of the inverse problem. 

Also let $\Tilde{h}$ be another minimum for $\mathcal{F}$. Then $\lambda(h) = \lambda(\Tilde{h})$ and $\pa_\nu u|_{\Gamma_D}(h) = \pa_\nu u|_{\Gamma_D}(\Tilde{h})$. Thus by Theorem \ref{uniqueness} we obtain $h = \Tilde{h}$. 
\end{proof}

In order to solve the minimization problem for $\mathcal{F}$ by using gradient methods, we compute the Fr\'echet derivative of the functional $\mathcal{F}$ with respect to $h$. It can be easily derived by Corollary \ref{differe H^2}. 
\begin{thm}
The Fr\'echet derivative of the functional $\mathcal{F}$ at the point $h \in \mathscr{A}$ in the direction $\xi$ is 
\begin{equation}\label{Fre deri}
    \mathcal{F'}(h)[\xi] = \int_\gamma \left\{ u(h) \varphi + (\lambda(h) - \lambda) u(h)^2 \right\}\xi \, ds, 
\end{equation}
where $\varphi$ is the solution of the following problem: 
\begin{equation}\label{adjoint problem}
\begin{cases}
-\Delta \varphi = \lambda(h) \varphi &\text{in} \,\, \Omega, \\
\varphi = \pa_\nu u(h) - g \hspace{-0.1cm} &\text{on} \,\, \Gamma_D, \\
 h \varphi +  \pa_\nu \varphi = 0 &\text{on} \,\, \gamma, \\
\displaystyle \int_\Omega u(h) \varphi \, dx = 0.
\end{cases}
\end{equation}
\end{thm}
\begin{rmk}
Note that \eqref{adjoint problem} is equivalent to the following problem: 
\begin{equation}\label{adj2}
\begin{cases}
(-\Delta - \lambda(h)) \varphi' = F &\text{in} \,\, \Omega, \\
\varphi' = 0 \hspace{-0.1cm} &\text{on} \,\, \Gamma_D, \\
 h \varphi' +  \pa_\nu \varphi' = 0 &\text{on} \,\, \gamma, \\
\displaystyle \int_\Omega u(h) \varphi' \, dx = 0,
\end{cases}
\end{equation}
for a suitable $F \in L^2(\Omega)$. Now, since $u(h)$ is the unique positive normalized eigenfunction of \eqref{P}, the only solution of \eqref{adj2} for $F\equiv 0$ is $\varphi' \equiv 0$. Then, by Fredholm alternative, for any $F \in L^2(\Omega)$ the problem admits a unique solution $\varphi'$, and so the same holds for \eqref{adjoint problem}.
\end{rmk}

\begin{proof}
By Corollary \ref{differe H^2} and the eigenvalue estimate \eqref{eigenvalue esti}, we obtain  
\begin{align*}
    &\mathcal{F}(h+\xi) - \mathcal{F}(h) \\
    &= \dfrac{1}{2} \norm{\pa_{\nu} u(h+\xi) - g}^2_{L^2(\Gamma_D)} + \dfrac{1}{2} \abs{\lambda(h+\xi) - \lambda}^2  \\
    &\qquad -\dfrac{1}{2} \norm{\pa_{\nu} u(h) - g}^2_{L^2(\Gamma_D)} - \dfrac{1}{2} \abs{\lambda(h) - \lambda}^2 \\
    &= \dfrac{1}{2} \norm{\pa_{\nu} u(h) + \pa_{\nu}u'(h)[\xi] + o(\norm{\xi}_{C^1(\gamma)}) - g}^2_{L^2(\Gamma_D)} \\
    &\qquad+ \dfrac{1}{2} \abs{\lambda(h)+\lambda'+o(\norm{\xi}^{3/2}_{C^1(\gamma)}) - \lambda}^2  -\dfrac{1}{2} \norm{\pa_{\nu} u(h) - g}^2_{L^2(\Gamma_D)} - \dfrac{1}{2} \abs{\lambda(h) - \lambda}^2 \\
    &= \int_{\Gamma_D} (\pa_\nu u(h) - g) \pa_\nu u'(h)[\xi] \, ds + \lambda'(\lambda(h) - \lambda) + o(\norm{\xi}_{C^1(\gamma)}). 
\end{align*}

Let us focus on the first term. By the Green's second identity we obtain 
\begin{align*}
    0 &= \int_\Omega \left( (-\Delta \varphi - \lambda(h) \varphi )u' - (- \Delta u' - \lambda(h) u' - \lambda' u(h)) \varphi \right) \, dx \\ 
    &= \int_\Omega (\varphi \Delta u' - u' \Delta \varphi) \, dx. 
\end{align*}
By the divergence theorem we have 
\begin{align*}
    0 &= \int_{\Gamma_D} \varphi \pa_\nu u' \, ds + \int_\gamma (\varphi \pa_\nu u' - u' \pa_\nu \varphi) \, ds \\
    &= \int_{\Gamma_D} \varphi \pa_\nu u' \, ds + \int_\gamma \left( \varphi(-hu' - \xi u(h)) + u'h\varphi \right) \, ds \\
    &= \int_{\Gamma_D} (\pa_\nu u(h) - g) \pa_\nu u' \, ds - \int_\gamma \xi u(h)\varphi \, ds. 
\end{align*}
Thus we obtain 
\begin{equation*}
    \int_{\Gamma_D} (\pa_\nu u(h) - g) \pa_\nu u' \, ds = \int_\gamma \xi u(h)\varphi \, ds. 
\end{equation*}
Therefore, since $\displaystyle \lambda' = \int_\gamma \xi u(h)^2 \, ds$, we have that the Fr\'echet derivative $\mathcal{F'}$ of the functional $\mathcal{F}$ is given by 
\begin{align*}
\mathcal{F'}(h)[\xi] &= \int_\gamma \xi u(h)\varphi \, ds + \lambda'(\lambda(h) - \lambda) \notag \\
&= \int_\gamma \left\{ u(h) \varphi + (\lambda(h) - \lambda) u(h)^2 \right\}\xi \, ds.  \qedhere
\end{align*}
\end{proof}

\section{Reconstruction algorithm and numerical tests}\label{sec of recon}
We use a gradient descent type algorithm to solve the minimization problem for the functional $\F$. Let $tol$ be a fixed tolerance level and $\tau_k >0$ the step sizes at each iteration $k$, that can be fixed or obtained by line search. In all numerical experiments below, we keep $\tau_k$ fixed.

\begin{algorithm}[H]
\caption{Reconstruction algorithm.} \label{algorithm}
Inputs: spectral data $(\lambda,g)$ and initial guess $h_0$. Set $k=0$ and iterate:
\begin{algorithmic}[1]
\State Compute the principal eigenfunction $u_k$ and eigenvalue $\lambda_k$ with Robin coefficient $h_k$, by solving Problem \eqref{P}.
\State Compute the solution $\varphi_k$ of the problem \eqref{adjoint problem}.
\State Compute the descent direction $\delta_k$ with the formula
\begin{equation}\label{direction}
\delta_k = -\left(u_k \varphi_k + (\lambda_k - \lambda) u_k^2\right).
\end{equation}
\State Define $h_{k+1} = h_k + \tau_k \delta_k$.
\State If $\|\delta_k\|_{C^1(\gamma)}> tol$, set $k = k+1$ and repeat.
\end{algorithmic}
\end{algorithm}

This will be used in the next numerical simulations, to show the effectiveness of the proposed reconstruction scheme. In what follows, we consider the annular region $\Omega = B(0,2) \setminus \overline{B(0,1)}$, with $\gamma = \partial B(0,1)$ and $\Gamma_D = \partial B(0,2)$. For each test, we create a mesh to generate the spectral data and a different one for the reconstruction.

The spectral data are obtained solving problem \eqref{P} with the target Robin coefficient by the Shift-invert method. We also add a uniform noise to the measurements. Given the noiseless boundary Neumann data $g = \partial_\nu u$ of the principal eigenfunction and its eigenvalue $\lambda$, the noisy data $\tilde g$ and $\tilde \lambda$ are obtained by adding to $g$ and $\lambda$ a uniform noise in the following way:
\begin{align*}
    \tilde g(x) &= g(x) + \varepsilon(x)\|g\|_{L^2(\Gamma_D)}, \quad x \in \gamma,\\
    \tilde \lambda &= \lambda(1 + \varepsilon_\lambda),
\end{align*}
where $\varepsilon(x)$ is a uniform random real in $(-\varepsilon_0,\varepsilon_0)$, and $\varepsilon_0,\varepsilon_\lambda>0$ are chosen according to the noise level. The relative noise on the Neumann data is measured as:
\[
\frac{\|\tilde g - g\|_{L^2(\Gamma_D)}}{\|g\|_{L^2(\Gamma_D)}},
\]
while the noise on the principal eigenvalue is simply $\varepsilon_\lambda$. We impose $\varepsilon_\lambda = \|\tilde g - g\|_{L^2(\Gamma_D)}/\|g\|_{L^2(\Gamma_D)}$ so that the two noise levels are comparable.

In case of noisy data, we use a regularized version of the reconstruction algorithm, based on the minimization of the Tikhonov regularized functional $\F_{\mathrm{reg}}$ \eqref{regularized cost functional}. In this case, the descent direction in algorithm \ref{algorithm} is computed with the following formula:
\begin{equation}
    \delta_k = -\left(u_k \varphi_k + (\lambda_k - \lambda) u_k^2 + \eta h_k\right),
\end{equation}
    where $\eta > 0$ is the regularization parameter, which is chosen experimentally.
    
All the computations are done using FreeFem++ \cite{hecht2012}.
\begin{figure}
\begin{picture}(420,420)
\put(-30,220){\includegraphics[width=7cm]{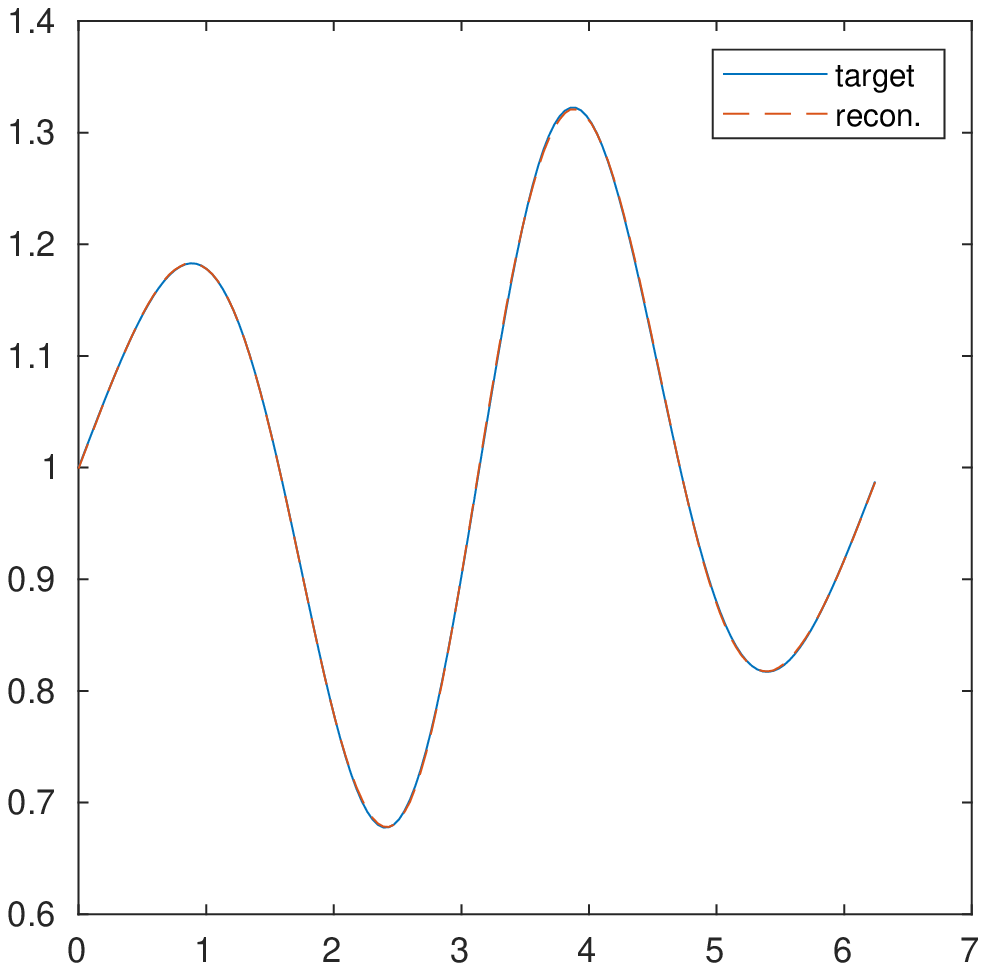}}
\put(200,220){\includegraphics[width=7cm]{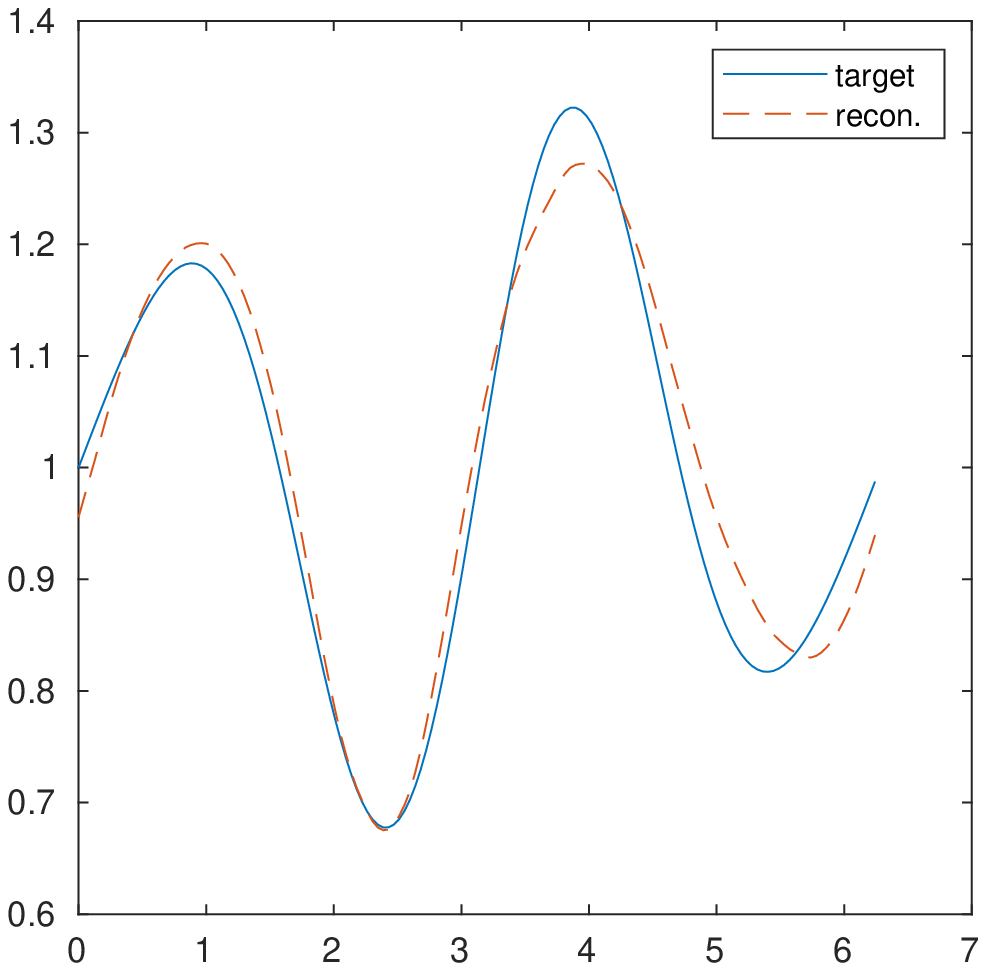}}
\put(-30,0){\includegraphics[width=7cm]{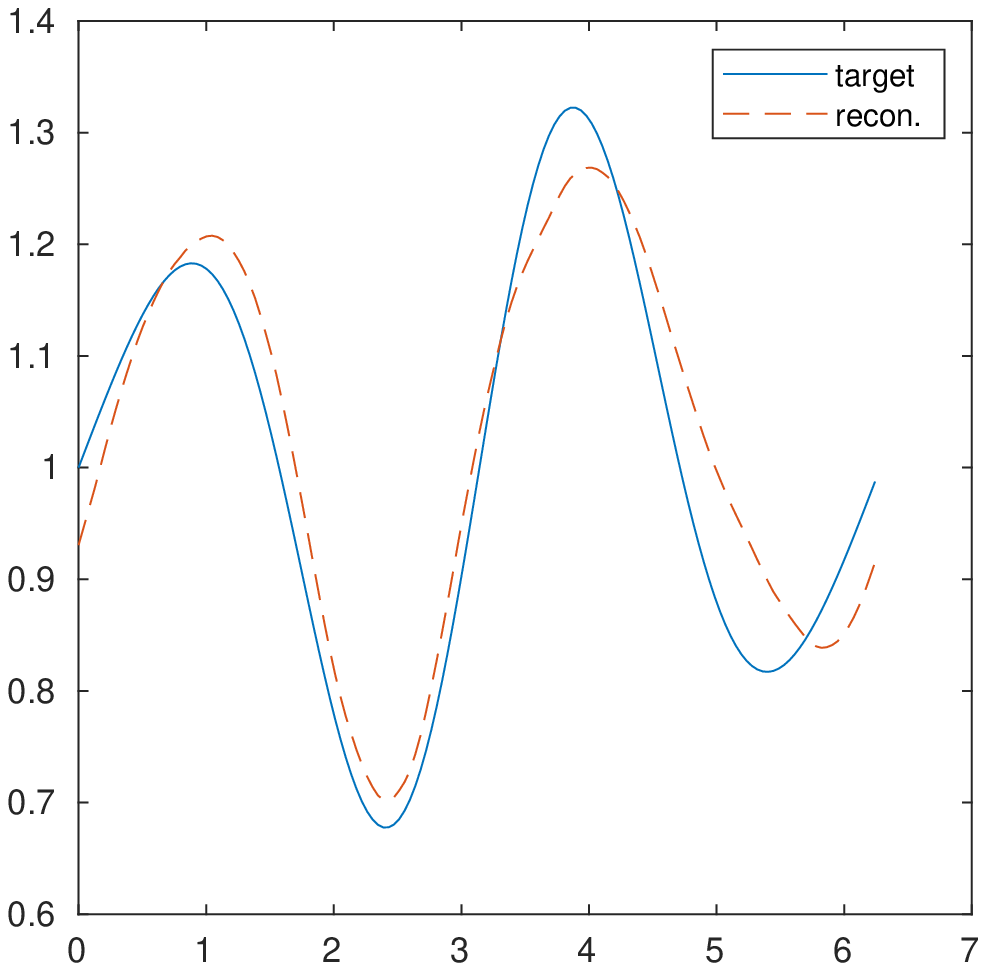}}
\put(200,0){\includegraphics[width=7cm]{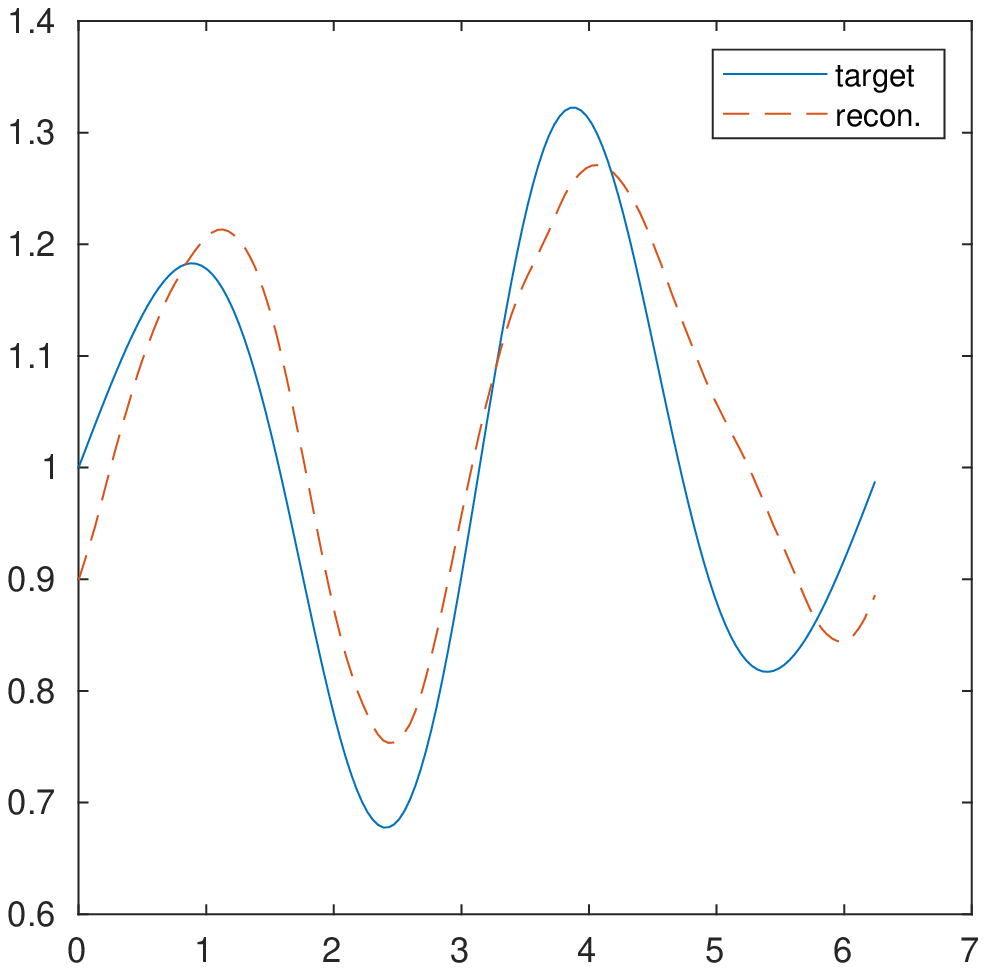}}
\put(-45,100){$h(\theta)$}
\put(-45,320){$h(\theta)$}
\put(185,100){$h(\theta)$}
\put(185,320){$h(\theta)$}
\put(70,-10){$\theta$}
\put(70,210){$\theta$}
\put(300,-10){$\theta$}
\put(300,210){$\theta$}
\end{picture}
\caption{\label{figure}Reconstruction of a Robin coefficient $h$ from spectral data for $\theta \in \gamma = \partial B(0,1)$. The blue line represents the target coefficient. The initial guess is $h \equiv 1$ in all cases. Top left: reconstruction from noiseless data. Top right: reconstruction from 0.5\% noisy data. Bottom left: reconstruction from 1\% noisy data. Bottom right: reconstruction from 2\% noisy data.}
\end{figure}

\smallskip
In Figure \ref{figure} we consider the reconstruction of the Robin coefficient 
$$h(x,y) = 1+\frac{xy}{2} - \frac{x^2y}{5}$$
for $(x,y) \in \gamma = \partial B(0,1)$, the interior part of the boundary of the annular region $\Omega$.

The initial guess is $h \equiv 1$ on $\gamma$. We present reconstruction from noiseless data (top left), and noisy data: 0.5 \% (top right), 1\% (bottom left) and 2\% (bottom right). The relative $L^2$ errors in the reconstructions are respectively $ 10^{-3}$, $3.7\cdot 10^{-2}$, $5.6\cdot 10^{-2}$, $8.7\cdot 10^{-2}$.

We can see that the algorithm performs well in case of no noise or low noise, while the quality of the reconstruction starts to deteriorate already for noise levels of $2\%$. This is completely coherent with the ill-posedness of the inverse problem. Though we have provided no theoretical evidence, it is reasonable to expect that the problem is severely ill-posed, since the similar inverse problem of the recovery of a Robin coefficient from a single boundary measurement is exponentially unstable \cite{sincich2007lipschitz}. In order to mitigate the instability, a different regularizer might be used, depending on a priori knowledge about $h$.

\section{Conclusions}\label{sec of conclusions}

We have presented uniqueness, local stability, and numerical reconstruction for the inverse problem of the recovery of a Robin coefficient from the Neumann data of the principal eigenfunction of the Laplacian with mixed boundary value problem together with the principal eigenvalue. To our knowledge, it is the first time that this problem has been considered, thought it has clear connections with coating and reinforcement problems.

Many questions are left open for future research, in particular extensions to less regular Robin coefficients and less regular domains. A thorough numerical study of the problem, namely regarding different regularization penalties or different reconstruction algorithms, is also left for future work.

\bibliographystyle{siam}
\bibliography{biblio}

\end{document}